\documentclass[11pt,amstex,amssymb]{amsart}
\usepackage{amsmath,amsthm,amsfonts,amssymb,amscd}
\usepackage[latin1]{inputenc}
\usepackage[all]{xy}
\usepackage[dvips]{graphicx}
\usepackage{amsmath}
\usepackage{amsthm}
\usepackage{amsfonts}
\usepackage{amssymb}
\newtheorem{theorem}{Theorem}

\newtheorem{claim}{Claim}

\newtheorem{corollary}{Corollary}

\newtheorem{definition}{Definition}
\newtheorem{example}{Example}

\newtheorem{lemma}{Lemma}

\newtheorem{proposition}{Proposition}
\newtheorem{remark}{Remark}

\subjclass[2000]{Primary 37F75, 57R30; Secondary 32M25, 32S65.}

\date{}

\keywords{Singular holomorphic
foliation,  holonomy group, closed orbit.}

\begin{document}

\title[Closed Orbits and Integrability]{Closed Orbits and Integrability for singularities of complex vector fields in dimension three}
\author{L. C\^amara and B. Sc\'ardua}
\date{}
\maketitle

\begin{abstract}
This paper is about the integrability of complex vector fields in dimension three in a neighborhood of a singular point. More precisely, we study the existence of holomorphic first integrals for isolated singularities of holomorphic vector fields in complex dimension three, pursuing the discussion started in  \cite{CaSc2009}. Under generic conditions, we
 prove a topological criteria for the existence of a holomorphic first integral. Our result may be seen as a kind of Reeb stability result for the framework of vector fields singularities in complex dimension three. As a consequence, we prove that, for the class of singularities we consider,  the
existence of a holomorphic first integral is invariant under topological equivalence.

\end{abstract}

\section{Introduction: Integrability, first integrals and closed orbits}
The problem of deciding whether a vector field or, more generally, an
ordinary differential equation can be integrated by studying its number
of non-transcendent solutions goes back to H. Poincar\'e, Dulac (\cite{Dulac}) and
other authors\footnote{Corresponding author: Bruno Sc\'ardua.  Tel/fax number: + 55 21 35 02 50 72; e-mail address: scardua@im.ufrj.br}. More recently the classical theorem of G. Darboux (\cite{jouanolou}) states that
a polynomial vector field in the complex plane admits a rational first integral
provided that if, and only if, it admits infinitely many algebraic solutions.
Of course the class of analytic equations is the one where the above problem
makes more sense. Moreover, with the arrival of the Theory of foliations
the use of geometrical/topological methods has given an important contribution
to the comprehension of the problem as well as some important results. Indeed,
a holomorphic vector field $X$ defined in a neighborhood $U\subset \mathbb C^n, n \geq 2$ of the origin $0\in\mathbb C^n$, with an isolated singularity at the origin, defines a germ of a one-dimensional holomorphic foliation
with a singularity at the origin in a natural way. Conversely, any germ of a holomorphic foliation with a singularity at the origin is defined in a small enough open neighborhood of the origin by holomorphic vector field with an isolated singularity at the origin. This is a consequence of Hartogs' extension theorem (\cite{gunning1}).

The local framework is not less important than the global (algebraic) case.
In this sense we have the important theorem of Mattei-Moussu (\cite{MaMo1980}) that states that
a germ of a holomorphic vector field at the origin of $\mathbb C^2$
 admits a holomorphic   first integral if, and only if,  the
following conditions are satisfied. (i) the leaves are closed off the origin
and (ii) only finitely many of them are separatrices, i.e.,  adhere to the origin.
Condition (ii) is usually known as \textit{non-dicriticity} of the (germ of a)
foliation induced by the (germ of a) vector field (\cite{CaSa}). A foliation
germ admitting a pure meromorphic first integral is necessarily dicritical. An
example of Suzuki shows then that there is no such a topological criteria for
existence of a meromorphic first integral (\cite{Suzuki}, \cite{Klughertz}).
Also interesting is the point of view adopted in
 \cite{Alexander-Verjovsky} where the authors prove, for a germ of a holomorphic vector field singularity in dimension $n \geq 2$, the existence of a holomorphic first integral,
 under the hypothesis of existence of an uniform bound for the volume of the orbits of the vector field,
 and some additional condition that restricts the  ``dicritical case".

Our goal is to investigate topological conditions assuring the existence of
holomorphic first integrals for vector field germs in dimension $3$. This is done in Theorem~\ref{topological criterion}. In few words, our result shows, for a generic class of singularities,  an equivalence between the existence of a holomorphic first integral and the existence of a suitable stable separatrix, and also with the existence of a suitable {\it flag}, i.e., a codimension one foliation containing the orbits of the vector field. Our result may be seen as a kind of Reeb stability theorem for singularities of complex vector fields.

According
to the above, we shall only consider the holomorphic, i.e., non-dicritical
case. Let us then introduce the  notation we use, already used in
\cite{CaSc2009}.
Denote the ring of germs of holomorphic functions on $(\mathbb{C}^{n},0)$ by
$\mathcal{O}_{n}$ and its maximal ideal by $\mathcal{M}_{n}$. Given a germ of
a holomorphic vector field $X\in\mathfrak{X}(\mathbb{C}^{n},0)$ we shall
denote by $\mathcal{F}(X)$ the germ of a one-dimensional holomorphic foliation
on $(\mathbb{C}^{n},0)$ induced by $X$.

\begin{definition}
[holomorphic first integral]\label{def. first int.}{\rm We say that a germ
of a holomorphic foliation $\mathcal{F}(X)$ \textit{has a holomorphic first
integral}, if there is a germ of a holomorphic map $F\colon(\mathbb{C}%
^{n},0)\rightarrow(\mathbb{C}^{n-1},0)$ such that:}

\begin{description}
\item[(a)] $\mathrm{F}$ {\rm is a submersion almost everywhere, i.e., if we
write } $\mathrm{F=(f}_{\mathrm{1}}\mathrm{,\cdots,f}_{\mathrm{n-1}}%
\mathrm{)}$ {\rm in coordinate functions, then the $(n-1)$-form }
$\mathrm{df}_{\mathrm{1}}\mathrm{\wedge\cdots\wedge df}_{\mathrm{n-1}}$
{\rm is non-identically zero, equivalently, it has maximal rank except for
a proper analytic subset;}

\item[(b)] {\rm The leaves of $\mathcal{F}(X)$ are contained in level
curves of $F$.}
\end{description}

{\rm Further, a germ $f$ of a meromorphic function at the origin
$0\in\mathbb{C}^{n}$ is called $\mathcal{F}(X)$\textit{-invariant} if the
leaves of $\mathcal{F}(X)$ are contained in the level sets of $f$. This can be
precisely stated in terms of representatives for $\mathcal{F}(X)$ and $f$, but
can also be written as $i_{X}(df)=X(f)\equiv0$. }
\end{definition}

Next we pass to describe the class of vector field germs we shall work with. A
germ of  a holomorphic vector field $X$ on $(\mathbb{C}^{n},0)$ is \textit{non-degenerate} if
its linear part $DX(0)$ is non-singular. As a linear map, generically $DX(0)$
has three distinct eigenvalues, thus is diagonalizable and $X$ has an isolated
singularity at the origin. From Poincar\'e-Dulac, Siegel and Brjuno
linearization theorems and from \cite{CaKuiPa1978}, \textit{generically}
(i.e., for a full measure subset of the set of the set of germs of holomorphic
vector fields), up to a change of coordinates,  the vector field $X$ leaves invariant the \textit{coordinate
hyperplanes} $x_{1}\cdots x_{n}=0$. This motivates the following definition:

\begin{definition}[generic germs]
{\rm We shall say that $\mathcal{F}(X)$ is \textit{non-degenerate generic}
if $DX(0)$ is non-singular, diagonalizable and, after some suitable change of coordinates,  $X$ leaves invariant the
coordinate planes.}
\end{definition}

Denote the set of germs of non-degenerate generic vector fields on
$(\mathbb{C}^{n},0)$ by $\operatorname{Gen}(\mathfrak{X}(\mathbb{C}^{n},0))$.
Let $X\in\operatorname{Gen}(\mathfrak{X}(\mathbb{C}^{n},0))$, $S$ a smooth
integral curve of $\mathcal{F}(X)$ through the origin, and $f$ a germ of an
$\mathcal{F}(X)$-invariant meromorphic function. Then we denote by
$\operatorname*{Hol}(\mathcal{F}(X),S,\Sigma)$ the holonomy of $\mathcal{F}%
(X)$ with respect to $S$ evaluated at a section $\Sigma$ transverse to $S$,
with $\Sigma\cap S=\{q_{\Sigma}\}$ a single point. Notice that we can choose
$\Sigma$ to be biholomorphic to a disc in ${\mathbb{C}}^{n-1}$ with center
corresponding to $q_{\Sigma}$. With this identification the group
$\operatorname*{Hol}(\mathcal{F}(X),S,\Sigma)$ is conjugate to a subgroup of
the group $\operatorname*{Diff}({\mathbb{C}}^{n-1},0)$ of germs of complex
diffeomorphisms fixing the origin in ${\mathbb{C}}^{n-1}$. A germ $f$ of a
meromorphic function at the origin $0\in\mathbb{C}^{n}$ is called
$\mathcal{F}(X)$- \textit{adapted} to $(\mathcal{F}(X),S)$ if it can be
written locally in the form $f=g/h$ where $g,h\in\mathcal{O}_{n}$ are
relatively prime, $S\subset Z(g)\cap Z(h)$, where $Z(g)$ and $Z(h)$ denote the
zero sets of $g$ and $h$ respectively, and $\left.  f\right\vert _{_{\Sigma}}$
is pure meromorphic for a generic transverse section $\Sigma$ to $S$. Given
vector field germs $X,Y\in\operatorname{Gen}(\mathfrak{X}(\mathbb{C}^{3},0))$
we have $\mathcal{F}(X)=\mathcal{F}(Y)$ if and only if for some nonvanishing
holomorphic function germ $u$ we have $Y=uX$. We shall then say that $X$ and
$Y$ are \textit{tangent}. Any vector field germ $X\in\operatorname{Gen}%
(\mathfrak{X}(\mathbb{C}^{3},0))$ admitting a holomorphic first integral must
satisfy the following condition (cf.\cite{CaSc2009}):

\begin{definition}
[condition $(\star)$]\label{definition:conditionstar} {\rm Let
$X\in\operatorname{Gen}(\mathfrak{X}(\mathbb{C}^{3},0))$. We say that $X$
satisfies condition $(\star)$ if there is a real line $L\subset\mathbb{C}$
through the origin, containing all the eigenvalues of $X$ and such that not
all the eigenvalues belong to the same connected component of $L\setminus
\{0\}$. }
\end{definition}

There is therefore one \textit{isolated} eigenvalue of $X$. The above
condition holds for $X$ if and only if holds for any vector field $Y$ such
that $X$ and $Y$ are tangent. Condition $(\star)$ implies that $X$ is in the
Siegel domain, but is stronger than this last. Denote by $\lambda(X)$ the
isolated eigenvalue of $X$ and by $S_{X}$ its corresponding invariant manifold
(the existence is granted by the classical invariant manifold theorem). We
call $S_{X}$ the \textit{distinguished axis} of $X$.
We shall say that  $X$ is  \textit{transversely stable}
with respect to $S_{X}$ if for any representative $X_{U}$ of the germ $X$,
defined in an open neighborhood $U$ of the origin, any open section
$\Sigma\subset U$ transverse to $S_{X}$ with $\Sigma\cap S_{X}=\{q_{\Sigma}%
\}$, and any open set $q_{\Sigma}\in V\subset\Sigma$ there is an open subset
$q_{\Sigma}\in W\subset V$ such that all orbits of $X_{U}$ through $W$
intersect $\Sigma$ only in $V$.

In this paper we prove the following topological criterion for the integrability of a
germ of a complex vector field singularity in dimension three:

\begin{theorem}
\label{theorem:partialtopological criterion}Suppose that $X\in\operatorname*{Gen}%
(\mathfrak{X}(\mathbb{C}^{3},0))$ satisfies condition $(\star)$ and let
$S_{X}$ be the distinguished axis of $X$. Then
$\mathcal{F}({X})$ has a holomorphic first integral if, and only if, the leaves of $\mathcal{F}({X})$ are closed off the singular set
$\operatorname*{Sing}(\mathcal{F}({X}))$ and transversely stable with respect
to $S_{X}$.
\end{theorem}

From this result we conclude  the invariance of the
existence of a holomorphic first integral for generic germs in dimension three, under topological
equivalence:

\begin{corollary}
\label{corollary:topological criterion} Let $X,Y\in\operatorname*{Gen}%
(\mathfrak{X}(\mathbb{C}^{3},0))$ be generic germs of holomorphic vector
fields, both satisfying condition $(\star)$. Assume that $X$ and $Y$ are
topologically equivalent. Then $X$ has a holomorphic first integral if and
only if $Y$ admits a holomorphic first integral.
\end{corollary}

Theorem~\ref{theorem:partialtopological criterion} above can be completed   (cf.
Theorem~\ref{topological criterion}), by weakening the topological hypothesis on the orbits, replacing the transverse stability by the existence of a suitable {\it flag}, i.e., a codimension one foliation, tangent to $\mathcal F(X)$.

\section{Finite orbits and periodic maps}

We determine the necessary conditions on the vector field $X\in
\operatorname*{Gen}(\mathfrak{X}(\mathbb{C}^{3},0))$ in order that
$\mathcal{F}({X})$ has a first integral.
Let $G\in\operatorname*{Diff}(\mathbb{C}^{2},0)$ and $V$ a neighborhood of the
origin where a representative (also denoted by $G$) of the germ $G$ is
defined. Then we denote by
\[
\mathcal{O}_{V}^{+}(G,x)=\left\{  G^{\circ
(n)}(x):\text{ }G^{\circ(j)}(x)\in V\text{, }j=0,\ldots,n\right\}
\]
 the so called
\textit{positive} \textit{semiorbit} of $x\in V$ by $G$. Analogously, the
\textit{negative semiorbit} of $x\in V$ by $G$ is the set $\mathcal{O}_{V}%
^{-}(G,x):=\mathcal{O}_{V}^{+}(G^{-1},x)$. The \textit{orbit} of $x\in V$ by
$G$ is the set $\mathcal{O}_{V}(G,x)=\mathcal{O}_{V}^{+}(G,x)\cup
\mathcal{O}_{V}^{-}(G,x)$. The cardinality of $\mathcal{O}_{V}(G,x)$ is
denoted by $|\mathcal{O}_{V}(G,x)|$.

\begin{theorem}
[Brochero Mart\'inez \cite{Bro2003}]\label{Brochero}Let $G\in
\operatorname*{Diff}(\mathbb{C}^{2},0)$, then the group generated by $G$ is
finite if and only if there exists a neighborhood $V$ of the origin such that
$|\mathcal{O}_{V}(G,x)|<\infty$ for all $x\in V$ and $G$ preserves infinitely
many analytic invariant curves at $0$.
\end{theorem}

Using the same arguments as in the one-dimensional case (cf. \cite{MaMo1980},
Proposition 1.1, p. 475-476), one can prove that a finite abelian (e.g.,
cyclic) subgroup of $\operatorname*{Diff}(\mathbb{C}^{n},0)$ is always
periodic, i.e., it is generated by a periodic (and linearizable) element.
Contrasting with the one dimensional case, in greater dimensions the
finiteness of the orbits in not enough to ensure the periodicity of the group
(cf. \cite{MaMo1980}, Theorem 2, p. 477).

\begin{example}
\label{translation} {\rm Consider the map $G(x,y)=(x+y^{2},y)$. The orbits
of $G$ are confined in the level set of $f(x,y)=y$ and are clearly finite.
Notice that $\#\mathcal{O}_{V}(G,(x,y))\rightarrow\infty$ as $y\rightarrow0$,
thus $G$ is not periodic nor linearizable. Furthermore, the orbits
$\mathcal{O}_{V}(G,(x,y))$ are far from being stable, since in each line
$(y=c)$ the map $G$ acts as a translation.}
\end{example}

We say that two germs of holomorphic functions $f,g\in\mathcal{O}_{2}$ are
\textit{generically transverse} if $df\wedge dg$ is not identically zero.

\begin{proposition}
\label{first int. x period.}Let $f,g\in\mathcal{O}_{2}$ be generically
transverse germs and $G\in\operatorname*{Diff}(\mathbb{C}^{2},0)$ be a complex
map germ having finite orbits and preserving the level sets of both $f$ and
$g$. Then $G$ is periodic.
\end{proposition}

\begin{proof}
The idea of the proof is the following: Since $f$ and $g$ are generically
transverse, then one can find a pure meromorphic function $h_{o}=f_{o}/g_{o}$
whose level sets are preserved by $G$. Hence the infinitely many curves
$f_{o}(x,y)-c\cdot g_{o}(x,y)=0$ with $c\in(\mathbb{C},0)$ pass through the
origin and are invariant by $G$. Thus Theorem \ref{Brochero} ensures that $G$
is periodic.

Now let us construct $h_{o}$. If $f/g$ is already pure meromorphic, then it is
enough to pick $h_{o}:=f/g$. Otherwise one has $f=h\cdot g^{k}$, where
$k\in\mathbb{Z}_{+}$, and $h$ is a germ of a holomorphic function not divisible
by $g$. Clearly, $h$ is $G$-invariant, thus if it has an irreducible component
distinct from the irreducible components of $g$, then $h/g$ must be a
$G$-invariant pure meromorphic function.

Suppose that the decomposition in irreducible components of $g$ and $h$ are of
the form $g=g_{1}^{p_{1}}\cdots g_{n}^{p_{n}}$ and $h=g_{1}^{q_{1}}\cdots
g_{n}^{q_{n}}$. Since $h$ is not divisible by $g$, then there must be
$j_{0}\in\{1,\cdots,n\}$ such that $q_{j_{0}}<p_{j_{0}}$. If there is also
$j_{1}\in\{1,\cdots,n\}$ such that $q_{j_{1}}>p_{j_{1}}$, then $h/g$ is a pure
meromorphic $G$-invariant function.

From now on we suppose that $q_{j}\leq p_{j}$ for all $j=1,\ldots,n$ with at
least one $j_{0}\in\{1,\cdots,n\}$ such that $q_{j_{0}}<p_{j_{0}}$. If there
is $j_{1}\in\{1,\cdots,n\}$ such that $q_{j_{1}}=p_{j_{1}}$, then after
reordering the indexes (if necessary) we may suppose that: (i) $q_{i}<p_{i}$
for all $i=1,\ldots,n_{0}$; (ii) $q_{i}=p_{i}$ for all $i=n_{0}+1,\cdots,n$;
for some $n_{0}\in\{1,\cdots,n-1\}$. Then $\overline{h}:=g/h=g_{1}%
^{p_{1}-q_{1}}\cdots g_{n_{0}}^{p_{n_{0}}-q_{n_{0}}}$ is a $G$-invariant germ
of a holomorphic function. Now, let $s_{1}:=\left[  p_{1}/(p_{1}-q_{1})\right]
+1$ (where $\left[  x\right]  $ denotes the integer part of $x\in\mathbb{R}$),
then a straightforward calculation shows that $g/\overline{h}^{s_{1}}$ is a
pure meromorphic function.

Hereafter we suppose that $q_{j}<p_{j}$ for all $j=1,\ldots,n$. Recall that
the Euclid's algorithm of a pair of positive integers $(p,q)$, $p>q$, is the
sequence of pairs of positive integers $\{(p_{j},q_{j})\}_{j=1}^{n+1}$ given
by: (1) $(p_{j+1},q_{j+1}):=(p,q)$; (2) $p_{j}=q_{j}\cdot r_{j}+s_{j}$, where
$r_{j}:=\left[  p/q\right]  $ and $s_{j}<q_{j}$; (3) $(p_{j+1},q_{j+1}%
):=(q_{j},r_{j})$; and (4) $s_{n}>0$ and $s_{n+1}=0$. This is called the
Euclid's sequence of the pair $(p,q)$. For simplicity, suppose that $g$ and
$h$ have only two irreducible components, say $g=f^{p}(\overline
{f})^{\overline{p}}$ and $h=f^{q}(\overline{f})^{\overline{q}}$, and let
$\{(p_{j},q_{j})\}_{j=1}^{n+1}$ and $\{(\overline{p}_{j},\overline{q}%
_{j})\}_{j=1}^{n+1}$ be the Euclid's sequence of $(p,q)$ and $(\overline
{p},\overline{q})$, respectively. If $r_{1}=[p_{1}/q_{1}]<[\overline{p}%
_{1}/\overline{q}_{1}]=\overline{r}_{1}$, then $p_{1}-(r_{1}+1)q_{1}<0$ and
$\overline{p}_{1}-(\overline{r}_{1}+1)\overline{q}_{1}\geq0$. If $\overline
{p}_{1}-(\overline{r}_{1}+1)q_{1}\neq0$, then $g/h^{r_{1}+1}$ is a
$G$-invariant germ of a pure meromorphic function, otherwise $g/h^{r_{1}%
+1}=1/f^{(r_{1}+1)q_{1}-p_{1}}$ and $g\cdot(g/h^{r_{1}+1})^{p_{1}}$ is a
$G$-invariant germ of a pure meromorphic function. Arguing inductively along the
Euclid's sequences of $(p,q)$ and $(\overline{p},\overline{q})$ one can always
construct a $G$-invariant pure meromorphic function unless $r_{j}=\overline
{r}_{j}$ for all $j=1,\cdots,n+1$. But this means that $(p,q)=(\alpha
s_{n},\beta s_{n})$ and $(\overline{p},\overline{q})=(\alpha\overline{s}%
_{n},\beta\overline{s}_{n})$ for some $\alpha,\beta\in\mathbb{Z}_{+}$.
Therefore $g$, $h$, and $f$ are powers of the same holomorphic function
$f^{s_{n}}(\overline{f})^{\overline{s}_{n}}$, thus $f$ and $g$ cannot be
generically transverse. A contradiction! The reasoning in the case of many
irreducible factors is analogous, being in fact a consequence of the above reasoning.
\end{proof}

A straightforward consequence is the following:

\begin{corollary}
Let $X\in\operatorname*{Gen}(\mathfrak{X}(\mathbb{C}^{3},0))$ and $S_{X}$ be
the distinguished axis of $X$. Suppose that $\mathcal{F}({X})$ admits a
meromorphic first integral, then the holonomy group $\operatorname*{Hol}%
(\mathcal{F}({X}),S_{X},\Sigma)$ is periodic.
\end{corollary}

\begin{example}
\label{blowing up G} {\rm Blowing up the diffeomorphism (cf.
\cite{Broetal2008}) $G=(g_{1},g_{2})=(x+y^{2},y)$ at the origin one has
\begin{align*}
\widetilde{G}(t,x)  &  =\left(  \frac{g_{2}(x,tx)}{g_{1}(x,tx)},g_{1}%
(x,tx)\right)  =\left(  \frac{t}{1+t^{2}x},x+tx\right) \\
&  =(t(1-t^{2}x+t^{4}x^{2}-t^{6}x^{3}+\cdots),x(1+t))\\
&  =(t-t^{3}x+t^{5}x^{2}-t^{7}x^{3}+\cdots,x+tx)
\end{align*}
whose orbits are finite and confined in the level sets of $\widetilde
{f}(t,x)=tx$. Further, $G$ acts in these level sets of $\widetilde{f}$ in some
sort of translation whose orbits increase in cardinality as $\widetilde
{f}(t,x)\rightarrow0$. In particular, Proposition \ref{first int. x period.}
ensures that $G$ does not preserve the level sets of a couple of generically
transverse functions $f,g\in\mathcal{O}_{2}$.}
\end{example}

\section{Closed leaves versus first integrals}

Now we construct an example showing that the closing of the leaves is not
sufficient to ensure the existence of first integrals for $\mathcal{F}({X})$
with $X\in\operatorname*{Gen}(\mathfrak{X}(\mathbb{C}^{3},0))$. The first
thing to be remarked is that the linear part of a generic vector field germ
having a first integral is determined by Proposition 1 in \cite{CaSc2009}. As
a consequence (cf. ï¿½2.3 in \cite{CaSc2009}) $\operatorname*{Hol}%
(\mathcal{F}({X}),S_{X},\Sigma)$ must be (a cyclic group generated by) a
resonant map preserving two smooth curves crossing transversely. In
particular, one cannot expect a map like the one in Example \ref{translation}
appearing as the (generator of the) holonomy of some $X\in\operatorname*{Gen}%
(\mathfrak{X}(\mathbb{C}^{3},0))$ with respect to $S_{X}$. Thus we blow up
such map and look to a neighborhood of the point determined by the exceptional
divisor and the strict transform of $(y=0)$.
Let $X\in\mathfrak{X}(\mathbb{C}^{3},0)$ be given by
\[
X(x)=-m_{1}[x_{1}(1+a_{1}(x))+x_{2}b_{1}(x)]\frac{\partial}{\partial x_{1}%
}-m_{2}x_{2}(1+a_{2}(x))\frac{\partial}{\partial x_{2}}+x_{3}\frac{\partial
}{\partial x_{3}}%
\]
where $m_{1},m_{2},k\in\mathbb{Z}_{+}$, $S:=(x_{1}=x_{2}=0)$ and
$\Sigma:=(x_{3}=1)$. Now consider the closed loop $\gamma:[0,1]\longrightarrow
S$ given by $\gamma(t)=(0,0,e^{2\pi\mathbf{i}t})$ and let $\overline{\Gamma
}_{(x_{1},x_{2})}(t)=(\Gamma_{1}(t,x_{1},x_{2}),\Gamma_{2}(t,x_{1}%
,x_{2}),\gamma(t))$ be its lifting along the leaves of $\mathcal{F}({X})$
starting at $(x_{1},x_{2},1)\in\Sigma$. In particular, the map $h\in
\operatorname*{Diff}(\mathbb{C}^{2},0)$ given by $\overline{\Gamma}%
_{(x_{1},x_{2})}(1)=(h(x_{1},x_{2}),1)$ is a generator of $\operatorname*{Hol}%
(\mathcal{F}({X}),S,\Sigma)$. Since $\overline{\Gamma}_{(x_{1},x_{2})}(t)$
belongs to a leaf of $\mathcal{F}({X})$, then
\[
\frac{\partial}{\partial t}\overline{\Gamma}_{(x_{1},x_{2})}(t)=\alpha
X(\Gamma_{1}(t,x_{1},x_{2}),\Gamma_{2}(t,x_{1},x_{2}),\gamma(t)).
\]
From this vector equation one has $\gamma^{\prime}(t)=\alpha\gamma(t)$, thus
$\alpha=2\pi\mathbf{i}$. Furthermore%
\begin{align}
\frac{\partial}{\partial t}\Gamma_{1}  &  =-2m_{1}\pi\mathbf{i[}\Gamma
_{1}\cdot(1+a_{1}(\Gamma_{1},\Gamma_{2},\gamma))+\Gamma_{2}\cdot b_{1}%
(\Gamma_{1},\Gamma_{2},\gamma)]\text{,}\label{eq1}\\
\frac{\partial}{\partial t}\Gamma_{2}  &  =-2m_{2}\pi\mathbf{i}\Gamma_{2}%
\cdot(1+a_{2}(\Gamma_{1},\Gamma_{2},\gamma))\text{.} \label{eq2}%
\end{align}

\begin{example}
{\rm  Let $X(x)=-[x_{1}+x_{2}^{2}b(x_{3})]\frac{\partial}{\partial x_{1}%
}-3x_{2}\frac{\partial}{\partial x_{2}}+x_{3}\frac{\partial}{\partial x_{3}}$,
then $S:=\{x_{1}=x_{2}=0\}$ is invariant by $X$ and the holonomy of
$\mathcal{F}({X})$ with respect to $S$ evaluated at $\Sigma=(x_{3}=1)$ has the
form $h=(h_{1},h_{2})$ with $h_{j}(x_{1},x_{2})=\Gamma_{j}(1,x_{1},x_{2})$,
where $\Gamma_{1}$ and $\Gamma_{2}$ satisfy respectively equations (\ref{eq1})
and (\ref{eq2}) above. Now if we let $\Gamma_{n}(t,x_{1},x_{2})=\sum
_{i+j\geq1}c_{i,j}^{n}(t)x_{1}^{i}x_{2}^{j}$, then (\ref{eq2}) is written in
the form }%
\[
\mathrm{\frac{\partial}{\partial t}\Gamma_{2}=-6\pi\mathbf{i}\Gamma_{2}.}%
\]
{\rm More precisely $\frac{d}{dt}c_{i,j}^{2}(t)=-6\pi\mathbf{i}\cdot
c_{i,j}^{2}(t)$, thus $c_{i,j}^{2}(t)=\lambda_{i,j}^{2}\exp(-6\pi\mathbf{i}t)$
for some $\lambda_{i,j}^{2}\in\mathbb{C}$. Since $\Gamma_{2}(0,x_{1}%
,x_{2})=x_{2}$, then $\lambda_{0,1}^{2}=1$ and $\lambda_{i,j}^{2}=0$
otherwise. Therefore $\Gamma_{2}(t,x_{1},x_{2})=\exp(-6\pi\mathbf{i}t)\cdot
x_{2}$ and $h_{2}(x_{1},x_{2})=x$}$_{2}${\rm . On the other hand,
(\ref{eq1}) is written in the form }%
\[
\mathrm{\frac{\partial}{\partial t}\Gamma_{1}=-2\pi\mathbf{i[}\Gamma
_{1}+e^{-6\pi\mathbf{i}t}x_{2}^{2}b(\gamma(t))]=-2\pi\mathbf{i(}\Gamma
_{1}+e^{-6\pi\mathbf{i}t}b(e^{2\pi\mathbf{i}t})\cdot x_{2}^{2}).}%
\]
{\rm Analogously, $\frac{d}{dt}c_{i,j}^{1}(t)=-2\pi\mathbf{i}\cdot
c_{i,j}^{1}(t)$ for all $(i,j)\neq(0,2)$. Since $\Gamma_{1}(0,x_{1}%
,x_{2})=x_{1}$, then $c_{1,0}^{1}(t)=\exp(-2\pi\mathbf{i}t)\cdot x_{1}$ and
$c_{i,j}^{1}(t)=0$ for all $(i,j)\notin\{(1,0),(0,2)\}$. Finally $\frac{d}%
{dt}c_{0,2}^{1}(t)=-2\pi\mathbf{i}(c_{0,2}^{1}(t)+e^{-6\pi\mathbf{i}%
t}b(e^{2\pi\mathbf{i}t}))$. Now recall that the solution to the Cauchy
problem
\[
\alpha^{\prime}(t)=-2\pi\mathbf{i}\cdot\alpha(t)-2\pi\mathbf{i}e^{-6\pi
\mathbf{i}t}b(e^{2\pi\mathbf{i}t})\text{, }\alpha(0)=0\text{.}%
\]
is given by%
\[
\alpha(t)=-2\pi\mathbf{i}e^{-2\pi\mathbf{i}t}\int_{0}^{t}e^{2\pi\mathbf{i}%
s}e^{-6\pi\mathbf{i}s}b(e^{2\pi is})ds=-e^{-2\pi\mathbf{i}t}\int_{0}%
^{t}e^{-6\pi\mathbf{i}s}b(e^{2\pi\mathbf{i}s})2\pi\mathbf{i}e^{2\pi
\mathbf{i}s}ds
\]
In particular, $\alpha(1)=-e^{-2\pi\mathbf{i}}\int_{\gamma}\frac{b(z)}{z^{3}%
}dz$. Thus, if we let $b(z)=-z^{2}/2\pi\mathbf{i}$, then $\alpha(1)=1$, and
$h(x_{1},x_{2})=(x_{1}+x_{2}^{2},x_{2})$.}
\end{example}

Completing the above example we obtain:

\begin{example}
\label{main example} {\rm  Consider the vector field $X(x,y,z)=-[x-\frac
{1}{2\pi\mathbf{i}}y^{2}z^{2}]\frac{\partial}{\partial x}-3y\frac{\partial
}{\partial y}+z\frac{\partial}{\partial z}$, after one blow up along the
$z$-axis one has
\begin{align*}
\pi^{\ast}X(t,x,z)  &  =-(x-\frac{1}{2\pi\mathbf{i}}t^{2}x^{2}z^{2}%
)\frac{\partial}{\partial x}+\frac{1}{x}(-3tx-t(-x-t^{2}x^{2}z^{2}%
))\frac{\partial}{\partial y}+z\frac{\partial}{\partial z}\\
&  =-x(1-\frac{1}{2\pi\mathbf{i}}t^{2}xz^{2})\frac{\partial}{\partial
x}-t(2-t^{2}xz^{2})\frac{\partial}{\partial x_{2}}+z\frac{\partial}{\partial
z}%
\end{align*}
which has an isolated singularity at the origin, and whose holonomy with
respect to the $z$-axis is precisely the map $\widetilde{G}$ in Example
\ref{blowing up G}. Thus it satisfies condition $(\star)$ and has all leaves
closed but does not admit a first integral in the sense of \cite{CaSc2009} (or
Definition \ref{def. first int.}).}
\end{example}

\section{The main result: stability, flags and first integrals. }

Example \ref{main example} shows that the statements of Theorems 1.2 and 1.3
in \cite{CaSc2009} are incomplete. The correct statements are thus suggested
by Proposition \ref{first int. x period.}. In one word, we need to consider stability.

\begin{definition}
[stability]{\rm A germ of a holomorphic vector field $X\in\operatorname*{Gen}%
(\mathfrak{X}(\mathbb{C}^{3},0))$ is said to be \textit{transversely stable}
with respect to $S_{X}$ if for any representative $X_{U}$ of the germ $X$,
defined in an open neighborhood $U$ of the origin, any open section
$\Sigma\subset U$ transverse to $S_{X}$ with $\Sigma\cap S_{X}=\{q_{\Sigma}%
\}$, and any open set $q_{\Sigma}\in V\subset\Sigma$ there is an open subset
$q_{\Sigma}\in W\subset V$ such that all orbits of $X_{U}$ through $W$
intersect $\Sigma$ only in $V$. }

{\rm A germ of a map $G\in\operatorname*{Diff}(\mathbb{C}^{2},0)$ is said to
be (topologically) \textit{stable} if for any representative $G\colon U \to
G(U)$, where $U$ is an open neighborhood of the origin, and any open set
$V\subset U$ there is an open subset $W\subset V$ such that $G^{\circ
(n)}(W)\subset V$ for all $n\in\mathbb{Z}$, i.e., all iterates of $G$ starting
in $W$ remain in $V$. }
\end{definition}

\begin{lemma}
\label{periodic orbits} Let the germ $G\in\operatorname*{Diff}(\mathbb{C}%
^{2},0)$ be represented by the map $G\colon W\rightarrow V$, where $W\subset
V$ are open neighborhoods of the origin with compact closure. Suppose $G$ has
finite orbits with stable positive semiorbit, i.e., there are $W$ and $V$ as
above with $W\subset V$ and satisfying $G^{\circ(n)}(x)\subset V$ for all
$x\in W$ and $n\in\mathbb{Z}_{+}$. Then $G$ is periodic, i.e., there is
$p\in\mathbb{Z}_{+}$ such that $G^{\circ p}=\operatorname*{id}$.
\end{lemma}

\begin{proof}
First notice that the topological hypothesis on the orbits of $G$ ensures that
these orbits are all periodic. Now consider the analytic set $C_{q}:=\{x\in
W\,:G^{\circ q}=x\}$, where $q\in\mathbb{Z}_{+}$. Then $C_{q}$ is a closed set
without interior points. Suppose that there is no $p\in\mathbb{Z}_{+}$ as
stated, then $W=%
{\textstyle\bigcup\nolimits_{q=1}^{\infty}}
C_{q}$. From Baire's category theorem, the set $W$ must have no interior point
also. A contradiction! The result then follows.
\end{proof}

Recall that a germ of a foliation $\mathcal{F}$ in $(\mathbb{C}^{2},0)$ has a
\textit{dicritical component} if there appears a dicritical singularity along
its resolution process or, equivalently, there are infinitely many leaves
adhering only at the origin (\cite{CaSa}). By a \textit{flag} containing
$\mathcal{F}(X)$, we mean a germ of a codimension one holomorphic foliation
$\mathcal{F}$ at $0\in\mathbb{C}^{3}$ with the property that (for some
representatives of each foliation defined in a common domain containing the
origin) each leaf of $\mathcal{F}(X)$ is contained in some leaf of
$\mathcal{F}$. The notion of flag is detailed in \cite{Mol}. As for our
purposes, assume that $\mathcal{F}$ is defined by a germ of an integrable
holomorphic one-form $\omega=Adx+Bdy+Cdz$ with a singularity at the origin. In
this case one writes $\mathcal{F}=\mathcal{F}_{\omega}$ and then the flag
condition can be stated as $i_{X}\omega\equiv0$. Given such a germ
$\mathcal{F}_{\omega}$, the singular set $\operatorname*{Sing}(\mathcal{F}%
_{\omega})$ is a codimension $\geq2$ germ of an analytic subset at the origin.
A codimension two irreducible component $K\subset\operatorname*{Sing}%
(\mathcal{F})\setminus\{0\}$ is a \textit{Kupka type} component if $d\omega$
does not vanish along $K$. According to Kupka's theorem (\cite{Omegar,
Kupka}), for a representative $\mathcal{F}_{U}$ of $\mathcal{F}$ in an open
neighborhood $0\in U$, where $\mathcal{F}$ is given by an integrable
holomorphic one-form $\omega_{U}$, and a representative $K_{U}\subset
\operatorname*{Sing}(\mathcal{F}_{U})$ of the component $K\subset
\operatorname*{Sing}(\mathcal{F}_{\omega})$, there is a plane foliation
$\eta(K)$ in a neighborhood of the origin $0\in\mathbb{C}^{2}$ such that for
each point $q\in K_{U}$ there is a holomorphic submersion $\varphi_{q}\colon
V_{q}\rightarrow\mathbb{C}^{2}$, with the property that $q\in V_{q}\subset
U,\,\varphi_{q}(q)=0$ and $\varphi_{q}^{\ast}(\eta(K))=\left.  \mathcal{F}%
_{U}\right\vert _{V_{q}}$. The foliation $\eta(K)$ is then called the
\textit{Kupka transverse type} of $\mathcal{F}$ along the Kupka component $K$.
One says that the Kupka component $K$ is \textit{dicritical} if the
corresponding transverse type $\eta(K)$ has a dicritical singularity at the
origin $0\in\mathbb{C}^{2}$. A particular case of a dicritical Kupka component
is the one given by $\mathcal{F}(X)\times\mathbb{C}$, where $\mathcal{F}(X)$
is a germ of a foliation in $(\mathbb{C}^{2},0)$ defined by a resonant
linearizable foliation in the Poincar\'e domain, i.e., in some appropriate
coordinate system $\mathcal{F}(X)$ is given by a vector field $X$ on
$(\mathbb{C}^{2},0)$ of the form $X=mx\frac{\partial}{\partial x}%
+ny\frac{\partial}{\partial y}$, where $m,n\in\mathbb{Z}_{+}$. Such kind of
dicritical Kupka component shall be called of \emph{radial type}.

Given a flag $\mathcal{F}_{\omega}$ containing the foliation $\mathcal{F}(X)$,
consider its restriction $\left.  \mathcal{F}_{\omega}\right\vert _{\Sigma}$
to a transverse section $\Sigma$ as above. Since $\Sigma$ is transverse to
$\mathcal{F}(X)$, it is also transverse to $\mathcal{F}_{\omega}$ off the
singular set $\operatorname*{Sing}(\mathcal{F}(X))$ and therefore one may
identify the germ of $\mathcal{F}_{\omega}$ at the point $q_{\Sigma}%
=\Sigma\cap S(X)$ with the germ of a foliation at the origin $0\in
\mathbb{C}^{2}$. One says then that $\left.  \mathcal{F}_{\omega}\right\vert
_{\Sigma}$ is dicritical if this corresponding germ in dimension two is dicritical.

Let $\mathcal{G}$ be a germ of a foliation in $(\mathbb{C}^{2},0)$ having a
dicritical component and $\phi\in\operatorname*{Diff}_{\operatorname*{id}%
}(\mathbb{C}^{2},0)$ be given by $\phi(x,y)=\exp[1]\widehat{X}(x,y)$ for a
(unique) formal vector field $\widehat{X}$ of order at least two (cf.
\cite{Broetal2008}, \cite{CaSc2011}) called the \textit{infinitesimal
generator} of $\phi$. Then one says that $\mathcal{G}$ is adapted to $\phi$ if
there is a resolution $\pi:(\mathcal{M},D)\longrightarrow(\mathbb{C}^{2},0)$
of $\widehat{X}$ such that $\pi^{\ast}(\mathcal{G})$ has infinitely many
curves transversal to $D$ (this happens precisely when we blow-up a dicritical
component of $\mathcal{G}$ along the resolution of $\widehat{X}$). In
particular, any dicritical $\mathcal{G}$ is automatically adapted to $\phi$.

Now recall that $X$\ has a linear part given by $J^{1}(X)=mx\frac{\partial
}{\partial x}+ny\frac{\partial}{\partial y}-kz\frac{\partial}{\partial z}$.
Thus the holonomy of $\mathcal{F}$\ with respect to the distinguished axis $z$
is periodic with linear part given by $J^{1}(h)(x,y)=(\exp(-\frac{2m\pi i}%
{k})x,\exp(-\frac{2n\pi i}{k})y)$; in particular $\phi:=h^{\circ(k)}$ is
tangent to the identity. Therefore, this map can be written locally in the
form $\phi(x,y)=\exp[1]\widehat{X}(x,y)$, where $\widehat{X}$ is its
infinitesimal generator. Then one says that $(\mathcal{F}({X}),\mathcal{F}%
_{\omega})$ is an \textit{adapted flag} if $\mathcal{F}({X})\subset
\mathcal{F}_{\omega}$ is a flag such that $\left.  \mathcal{F}_{\omega
}\right\vert _{\Sigma}$ is a germ of a foliation having a dicritical component
and adapted to $\phi=h^{\circ(k)}$. In particular, if $\left.  \mathcal{F}%
_{\omega}\right\vert _{\Sigma}$ is dicritical, then $(\mathcal{F}%
({X}),\mathcal{F}_{\omega})$ is automatically an adapted flag.

Notice that the last definitions are of finite determinacy character.

Using this terminology, one may correctly restate the main result of
\cite{CaSc2009} as follows.

\begin{theorem}
\label{topological criterion}Suppose that $X\in\operatorname*{Gen}%
(\mathfrak{X}(\mathbb{C}^{3},0))$ satisfies condition $(\star)$ and let
$S_{X}$ be the distinguished axis of $X$. Then the following conditions are equivalent:

\begin{description}
\item[(1)] The leaves of $\mathcal{F}({X})$ are closed off the singular set
$\operatorname*{Sing}(\mathcal{F}({X}))$ and transversely stable with respect
to $S_{X}$;

\item[(2)] $\operatorname*{Hol}(\mathcal{F}({X}),S_{X},\Sigma)$ has finite
orbits and is (topologically) stable;

\item[(3)] $\operatorname*{Hol}(\mathcal{F}({X}),S_{X},\Sigma)$ is periodic
{\rm (}in particular linearizable and finite{\rm )};

\item[(4)] $\mathcal{F}({X})$ has a holomorphic first integral.
\end{description}

Moreover, in terms of flags of foliations, the above conditions are also
equivalent to each of the following conditions:

\begin{description}
\item[(5)] The leaves of $\mathcal{F}({X})$ are closed off
$\operatorname*{Sing}(\mathcal{F}({X}))$ and there is an adapted flag
$(\mathcal{F}({X}),\mathcal{F}_{\omega})$;

\item[(6)] The leaves of $\mathcal{F}({X})$ are closed off
$\operatorname*{Sing}(\mathcal{F}({X}))$ and there is a flag $\mathcal{F}%
({X})\subset\mathcal{F}_{\omega}$ such that $\mathcal{F}_{\omega}$ is a Kupka
component of radial type.
\end{description}
\end{theorem}

\begin{proof}
[Proof of the first part of Theorem \ref{topological criterion}]It follows
immediately from the definition of transverse stability of germs of vector
fields and from (topological) stability of maps that (1) implies (2). It comes
from Lemma \ref{periodic orbits} that (2) implies (3). Now let us prove that
(3) implies (4). Since $X$ satisfies condition $(\star)$ and
$\operatorname*{Hol}(\mathcal{F}({X}),S_{X},\Sigma)$ is linearizable, then
\cite{EliYa1984} ensures that $\mathcal{F}({X})$ is linearizable. Therefore
one may suppose without loss of generality that $X(x)=\lambda x_{1}%
\frac{\partial}{\partial x_{1}}+\mu x_{2}\frac{\partial}{\partial x_{2}%
}-\kappa x_{3}\frac{\partial}{\partial x_{3}}$, where $\lambda,\mu,\kappa
\in\mathbb{R}_{+}$. Since $\operatorname*{Hol}(\mathcal{F}({X}),S_{X},\Sigma)$
is periodic one may suppose without loss of generality that $\lambda
=m,\mu=n,\kappa=k\in\mathbb{Z}_{+}$. The result then follows from Lemma~2.3 in
\cite{CaSc2009}. Finally let us verify that (4) implies (1). The existence of
a first integral for $\mathcal{F}({X})$ ensures that the leaves of
$\mathcal{F}({X})$ are closed off $\operatorname*{Sing}(\mathcal{F}({X}))$.
Furthermore, $\operatorname*{Hol}(\mathcal{F}({X}),S_{X},\Sigma)=\left\langle
H\right\rangle $ admits a couple of generically transverse $\mathcal{F}({X}%
)$-invariant holomorphic functions whose restrictions to $\Sigma$ have the
level sets preserved by $H$. Thus Proposition \ref{first int. x period.}
ensures that $\operatorname*{Hol}(\mathcal{F}({X}),S_{X},\Sigma)$ is periodic
and, in particular, topologically stable. Hence the leaves of $\mathcal{F}%
({X})$ are transversely stable with respect to $S_{X}$. This proves the first
four equivalences in Theorem~\ref{topological criterion}
\end{proof}

As a straightforward consequence, one has the following topological criterion
for the existence of invariant meromorphic functions for elements in
$\operatorname*{Gen}(\mathfrak{X}(\mathbb{C}^{3},0))$.

\begin{theorem}
\label{F-invariant merom.}Le $X\in\operatorname*{Gen}(\mathfrak{X}%
(\mathbb{C}^{3},0))$ satisfy condition $(\star)$, and $S_{X}$ be its
distinguished axis. Suppose that $\mathcal{F}({X})$ has closed leaves off
$\operatorname*{Sing}(${$\mathcal{F}({X})$}$)$ and is transversely stable with
respect to $S_{X}$, then there is an $\mathcal{F}({X})$-invariant meromorphic
function adapted to $(\mathcal{F}({X}),S_{X})$.
\end{theorem}

\begin{proof}
The proof of Theorem \ref{F-invariant merom.} is now almost identical to the
original version (Theorem~ 2 in \cite{CaSc2009}), only including the stability
hypothesis in the proof to obtain then a holomorphic first integral and after
the desired $\mathcal{F}({X})$-invariant meromorphic function adapted to
$(\mathcal{F}({X}),S_{X})$.
\end{proof}


Now we study the topological invariance of the existence of a holomorphic
first integral for a generic germ of a holomorphic vector field, as a
consequence of our preceding results. We recall that two germs of holomorphic
vector fields $X$ and $Y$ at the origin $0\in\mathbb{C}^{n}$ are
\textit{topologically equivalent} if there is a homeomorphism $\psi\colon
U\rightarrow V$ where $U,V$ are neighborhoods of the origin $0\in
\mathbb{C}^{n}$, where $X$ and $Y$ have representatives $X_{U}$ and $Y_{V}$
respectively, such that $\psi$ takes orbits of $X_{U}$ into orbits of $Y_{V}$.
Such a map $\psi$ takes separatrices of $X_{U}$ into separatrices of $Y_{V}$:
indeed, a separatrix of $X_{U}$ is an orbit which is closed off the origin,
and the same holds for its image under $\psi$. Assume that the vector field
$X$ is generic satisfying condition $(\star)$ and admits a holomorphic first
integral. In this case one has:

\begin{claim}
\label{claim:separatrix} The vector field $X$ is analytically linearizable,
say $X(x,y,z)=X_{n,m,-k}:=nx\frac{\partial}{\partial x}+my\frac{\partial
}{\partial y}-k\frac{\partial}{\partial z}$ with $n,m,k\in\mathbb{Z}_{+}$ and
suitable local coordinates $(x,y,z)\in(\mathbb{C}^{3},0)$. In particular, $X$
admits a unique separatrix off the dicritical plane, and this separatrix
corresponds to the distinguished separatrix $S_{X}$.
\end{claim}

\begin{proof}
Indeed, the analytic linearization of $X$ is a straightforward consequence of
the first part of Theorem~\ref{topological criterion} (or, since by hypothesis
there is a holomorphic first integral, in view of
Proposition~\ref{periodic orbits} and \cite{EliYa1984}). In this normal form
\[
X(x,y,z)=X_{n,m,-k}:=nx\frac{\partial}{\partial x}+my\frac{\partial}{\partial
y}-k\frac{\partial}{\partial z}%
\]
the \textquotedblleft dicritical plane" is the plane $\{z=0\}$ and the
distinguished separatrix is the $z$-axis. The orbit $\mathcal{O}_{(a,b,c)}$ of
$X$ through the point $(a,b,c)$ is given by
\[
\phi(t)=(x(t),y(t),z(t))=(ae^{nt},bye^{mt},ce^{-kt}),t\in\mathbb{C}.
\]

Thus, if $\mathcal{O}_{(a,b,c)}$ accumulates at the origin, either $c=0$ or
$c\neq0$ and $a=b=0$. For instance, if $c\neq0\neq a$, then the orbit is
contained in the hypersurface $x^{k}z^{n}=a^{k}c^{n}\neq0$, which does not
accumulate at the origin.
\end{proof}

\begin{lemma}
\label{lemma:topseparatrix} A topological equivalence takes the distinguished
axis of $X$ into the distinguished axis of $Y$.
\end{lemma}

\begin{proof}
Indeed, as we have seen above, the image $\psi(S_{X})$ is some separatrix of
$Y$. If this is not the distinguished axis of $Y$, then the distinguished axis
of $Y$ is taken by $\psi^{-1}$ into a separatrix other than the distinguished
axis of $X$. Therefore, according to Claim~\ref{claim:separatrix}, $\psi
^{-1}(S_{Y})$ must be a separatrix of the \textquotedblleft
dicritical\textquotedblright\ part of $X$, i.e., in the coordinates $(x,y,z)$
above, where $X(x,y,z)=X_{n,m,-k}$, we have $\psi^{-1}(S_{Y})\subset\{z=0\}$.
Nevertheless, any invariant neighborhood of a leaf contained in a dicritical
separatrix of $X$ off the origin intersects infinitely many separatrices
(namely, those contained in the intersection of this neighborhood with the
dicritical plane $\{z=0\}$). On the other hand, this same phenomena does not
occur for arbitrarily small invariant neighborhoods of a leaf contained in the
distinguished axis $S_{Y}$ of $Y$. Therefore, necessarily $\psi(S_{X})$ is the
distinguished axes of $Y$.
\end{proof}

From the above considerations we  immediately obtain Corollary~\ref{corollary:topological criterion} from Theorem~\ref{topological criterion}.

\section{Flags and integrability}

In this section, the second part of the of the main result is proved, i.e.,
the equivalence of the first four with the final two equivalences in
Theorem~\ref{topological criterion}.

In fact, our main concern here is the following: given a germ of a foliation by
curves $\mathcal{F}$ induced by a germ of a vector field of the form%
\begin{equation}
X=mx(1+a(x,y,z))\frac{\partial}{\partial x}+ny(1+b(x,y,z))\frac{\partial
}{\partial x}-kz(1+c(x,y,z)\frac{\partial}{\partial z} \label{eq4}%
\end{equation}
with $a,b,c\in\mathcal{M}_{3}$, study the consequences of the existence of a
codimension $1$ germ of a holomorphic foliation tangent to $X$, which is
transversely dicritical with respect to $S$.

\subsection{Transverse structure}

We begin by studying the consequences of the existence of a flag foliation
with a dicritical transverse type for a vector filed $X\in\operatorname{Gen}%
(\mathfrak{X}(\mathbb{C}^{3},0))$.

\begin{lemma}
\label{transversal dicrit.}Let $\mathcal{F}$ be a germ of a foliation by
curves on $(\mathbb{C}^{3},0)$, $S$ an invariant curve of $\mathcal{F}$
through the origin and $\mathcal{G}$ a codimension one foliation satisfying
the following conditions:

\begin{description}
\item[(i)] $\mathcal{G}$ is tangent to $X$;

\item[(ii)] There is a section $\Sigma$ transverse to $S$ such that $\left.
\mathcal{G}\right\vert _{\Sigma}$ is dicritical.
\end{description}

Then $\mathcal{G}$ is transversely dicritical with respect to $S.$
\end{lemma}

\begin{proof}
Since the orbits of $X$ are contained in the leaves of $\mathcal{G}$, then
they are invariant by the flow of $X$. Therefore if $\Sigma^{\prime}$ is
another section transversal to $S$ and $\phi:\Sigma\longrightarrow
\Sigma^{\prime}$ is an element of the holonomy pseudogroup of $X$ with respect
to $S$, then it is a diffeomorphism taking the leaves of $\left.
\mathcal{G}\right\vert _{\Sigma}$ onto the leaves of $\left.  \mathcal{G}%
\right\vert _{\Sigma^{\prime}}$.
\end{proof}

The following result leads to deep implications in the transverse dynamics of
the foliation $\mathcal{F}(X)$ in the presence of a dicritical flag
$\mathcal{F}(X)\subset\mathcal{F}_{\omega}$. First recall some notation. Let
$\operatorname*{Diff}_{\operatorname*{id}}(\mathbb{C}^{2},0)\subset
\operatorname*{Diff}(\mathbb{C}^{2},0)$ denote the group of germs of
diffeomorphism tangent to the identity. Further, let $\mathcal{F}_{\omega}$ be
a germ of a foliation in $(\mathbb{C}^{2},0)$ given by $\omega=0$, then we
denote by $\operatorname*{Aut}(\mathcal{F}_{\omega})$ the subgroup of
$\operatorname*{Diff}(\mathbb{C}^{2},0)$ given by those $\phi\in
\operatorname*{Diff}(\mathbb{C}^{2},0)$ preserving $\mathcal{F}_{\omega}$,
i.e., such that $\phi^{\ast}\omega\wedge\omega=0$.

\begin{lemma}
\label{fin. orb. dicr. tang. id.}Let $\mathcal{F}_{\omega}$ be a germ of
a foliation in $(\mathbb{C}^{2},0)$ having a dicritical component and adapted to
$\phi\in\operatorname*{Diff}_{\operatorname*{id}}(\mathbb{C}^{2},0)$. Thus
$\phi\in\operatorname*{Aut}(\mathcal{F}_{\omega})$ and has finite orbits if
and only if $\phi$ is the identity.
\end{lemma}

\begin{proof}
Let $\pi:(\mathcal{M},D)\rightarrow(\mathbb{C}^{2},0)$ be the resolution of
$\phi$ introduced in \cite{Abate2001}, $\widetilde{\mathcal{F}}:=\pi^{\ast
}\mathcal{F}$ the strict transform of $\mathcal{F}$ via $\pi$, and
$\widetilde{\phi}$ the lifting of $\phi$. Since $\phi\in\operatorname*{Diff}%
_{\operatorname*{id}}(\mathbb{C}^{2},0)$, then $\widetilde{\left.
\phi\right\vert }_{D}=\left.  \operatorname*{id}\right\vert _{D}$. If
$\widetilde{\mathcal{F}}_{j}\subset\widetilde{\mathcal{F}}$ is a dicritical
component of $\widetilde{\mathcal{F}}$ defined in a neighborhood of the
irreducible component $D_{j}\subset D$, then it is given in appropriate
coordinate systems by a fibration transversal to $D_{j}$, up to a finite
number of singular leaves or smooths leaves tangent to $D_{j}$. More
precisely, there is an open set $U_{j}:=D_{j}\backslash\{p_{1},\cdots,p_{r}\}$
such that $\widetilde{\left.  \phi\right\vert }_{U_{j}}$ can be seen as a
familly of germs of automorphisms of $(\mathbb{C},0)$ with parameters in
$U_{j}\subset D_{j}\simeq\mathbb{CP}^{1}$ (see Figure $1$). Let $\widetilde
{\phi}_{t}\in\operatorname*{Diff}(\mathbb{C},0)$ be given by $\widetilde{\phi
}_{t}(x):=\widetilde{\phi}(t,x)$ for some $t\in U_{j}$, then the classical
Leau-Fatou flower theorem says that $\widetilde{\phi}_{t}$ has a parabolic
fixed point at the origin, unless it is the identity. The result then follows
by analytic continuation.
\end{proof}

%

\begin{center}
\includegraphics[
height=0.7939in,
width=2.5365in
]%
{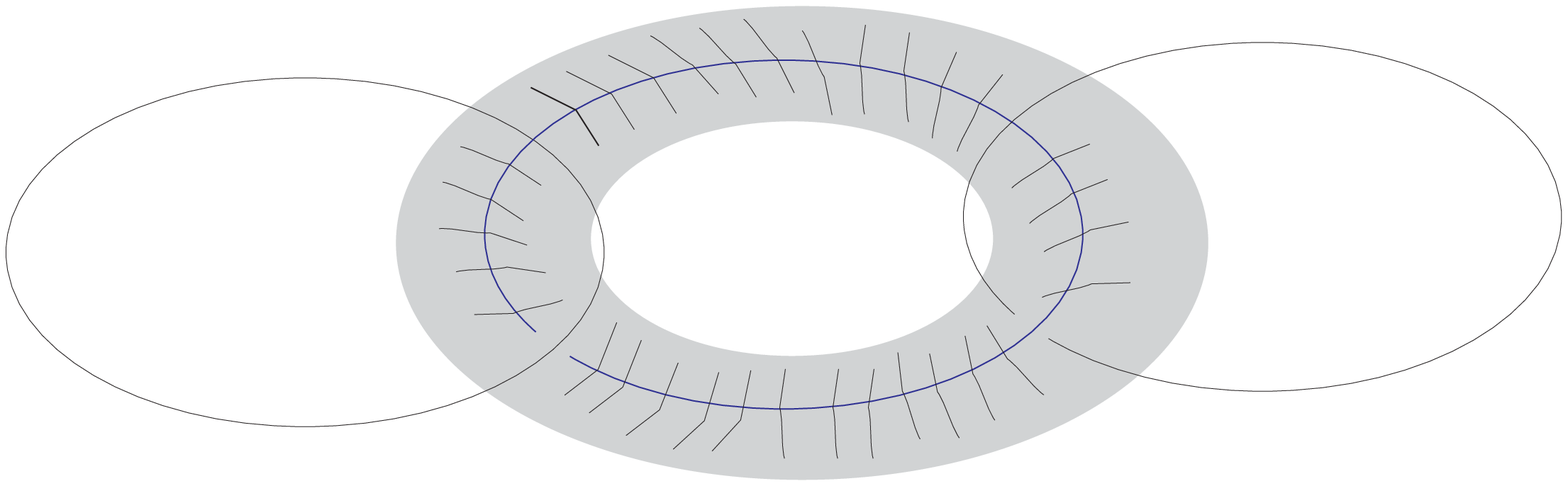}%
\\
Figure 1. A dicritical component of $\widetilde{\mathcal{F}}$.
\end{center}

\subsection{The existence of an algebraic-topological criterion}

Here shall finish the proof of Theorem \ref{topological criterion}. For this
sake, let us first recall some facts proved along this work and introduce some
terminology. First notice that any $X\in\operatorname*{Gen}(\mathfrak{X}%
(\mathbb{C}^{3},0))$ admitting a holomorphic first integral must satisfy
condition $(\star)$ in Definition~\ref{definition:conditionstar} (cf. \cite{CaSc2009}).
Assume the curve $S_{X}$ is the $z$-axis, let $\Sigma_{z}%
:=(z=\operatorname*{const}.)$ be a section transverse to $S_{X}$, and
$\operatorname*{Hol}(\mathcal{F}(X),S_{X},\Sigma_{z})$ be the holonomy of
$\mathcal{F}(X)$ with respect to $S_{X}$ evaluated at $\Sigma_{z}$.

\begin{proof}
[End of the proof of Theorem \ref{topological criterion}]First suppose all the
leaves of $\mathcal{F}(X)$ are closed off $\operatorname*{Sing}({\mathcal{F}%
(X)})=\{0\}\subset\mathbb{C}^{3}$ and there is an adapted flag $\mathcal{F}%
({X})\subset\mathcal{F}_{\omega}$. Given a leaf $L$ of $\mathcal{F}(X)$ it
follows that the closure $\overline{L}\subset L\cup\operatorname*{Sing}%
(\mathcal{F}(X))$ is an analytic subset of pure dimension one (\cite{GuRo}) in
$\mathbb{C}^{3}$. Since this leaf is transversal to $\Sigma_{z}$, one
concludes that $\overline{L}\cap\Sigma_{z}$ is a finite set. On the other
hand, given a point $x\in L\cap\Sigma_{z}$, its orbit in the holonomy group is
also contained in $L\cap\Sigma_{z}$, so that it is a finite set. Thus the
orbits of $H_{z}$ are finite. By hypothesis, for any $z_{0}\in S(X)$ the
foliation $\left.  \mathcal{F}_{\omega}\right\vert _{\Sigma_{z_{0}}}$ has a
dicritical component. Now consider a simple loop $\gamma$ around the origin,
inside the $z$-axis, starting from $z_{0}$. Pick a leaf $L$ of $\left.
\mathcal{F}_{\omega}\right\vert _{\Sigma_{z_{0}}}$ and consider the liftings
of $\gamma$ starting at points of $L$, along the trajectories of
$\mathcal{F}(X)$. Then these liftings form a three dimensional real variety,
say $S_{L}$, whose intersection with $\Sigma_{z_{0}}$ is given by $L$ and
$L^{\prime}$ (see Figure $2$). In particular if $h:=h_{\gamma}$ is the
generator of $\operatorname*{Hol}{}(\mathcal{F}(X),S_{X},\Sigma_{z})$, then
$L^{\prime}=h(L)$. For the one-form $\omega$, one has that $S_{L}$ is tangent
to $\operatorname*{Ker}(\omega)$, and $S_{L}\cap$ $\Sigma_{z_{0}}$ is tangent
to the induced foliation $\left.  \mathcal{F}_{\omega}\right\vert
_{\Sigma_{z_{0}}}$. Thus $L^{\prime}$ is a leaf of $\left.  \mathcal{F}%
_{\omega}\right\vert _{\Sigma_{z_{0}}}$. Since $\left.  \mathcal{F}_{\omega
}\right\vert _{\Sigma_{z_{0}}}$ has a dicritical component and $h$ is a
difeomorphism with resonant linear part having finite orbits, then Lemma
\ref{fin. orb. dicr. tang. id.} ensures that $h$ is periodic (in particular
linearizable and finite). Since $\mathcal{F}(X)\in\operatorname*{Gen}%
(\mathfrak{X}(\mathbb{C}^{3},0))$ has linearizable periodic holonomy, then it
follows from \cite{EliYa1984} that the foliation $\mathcal{F}(X)$ is also
analytically linearizable. Therefore, one may suppose without loss of
generality that $X(x,y,z)=mx\frac{\partial}{\partial x}+ny\frac{\partial
}{\partial y}-kz\frac{\partial}{\partial z}$. This vector field has a
holomorphic first integral. From the above linearization, it is easy to see
that the flag foliation $\mathcal{F}_{\omega}$ containing $\mathcal{F}(X)$
must have a linear dicritical Kupka transverse type along the $z$-axis. In
particular, $\mathcal{F}_{\omega}$ is of radial type. This proves that (5)
implies (1)-(4) and also (6). Since the converse is immediate, this proves
that the first four conditions in Theorem~\ref{topological criterion} are
equivalent to conditions (5) and (6).
\end{proof}

%

\begin{center}
\includegraphics[
height=1.9406in,
width=3.039in
]%
{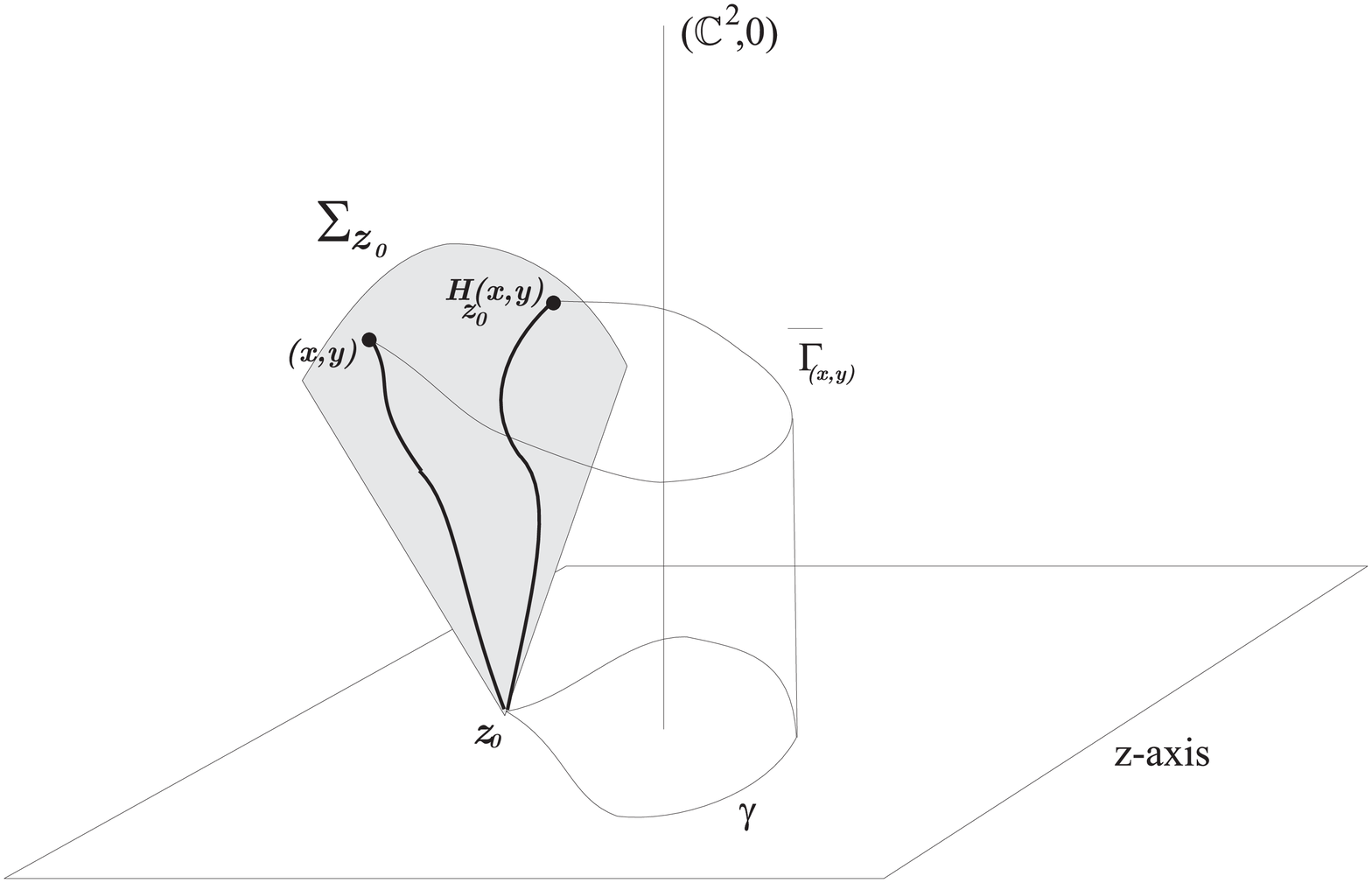
}%
\\
Figure 2. The lifting of $\gamma$ along the leaves of $\mathcal{F}$ starting
at points of $L$.
\end{center}

\vglue.3in

{\small
\begin{remark}[Parabolic curves and smooth sets of fixed points cf. \cite{Abate2001}]
{\rm  In \cite{CaSc2009} it is stated an integrability result, mentioning only the fact that the leaves of $\mathcal{F}(X)$ are closed off $\operatorname*{Sing}%
(\mathcal{F}(X))$.
Nevertheless, as we shall see in the next sections, this result is not correct.
Indeed, there are such vector fields without holomorphic first integral (cf.
Example~\ref{main example}).

  Let us identify precisely the missing point in
\cite{CaSc2009} and to determine some further topological conditions in order
to correct the statements of the main theorems therein (Theorems 1.2 and 1.3
in \cite{CaSc2009}).
Along these lines, we shall keep all the notations introduced in
\cite{CaSc2009}. First, let us deal with the missing point in \cite{CaSc2009}.
In Theorem 3.6 of \cite{CaSc2009}, we have stated that every non trivial
complex map germ fixing the origin admits a parabolic curve. Javier Ribon draw
our attention to the fact that this is not true with the following example:

{\rm Let $X^{o}=py\frac{\partial}{\partial y}-qx\frac{\partial}{\partial
x}$ with $p,q\in\mathbb{Z}_{+}$ and $X=xyX^{o}$, then the orbits of the map
$\Phi(x,y)=\exp[1]X(x,y)$ are confined in the level sets of the first integral
$f(x,y)=x^{p}y^{q}$ to the vector filed $X$. Therefore $\Phi$ has no orbit
attracting to the origin, and thus does not admit any parabolic curve at the
origin.}

Some time after that, Marco Abate communicated us the same fact showing that
Theorem 3.6 in \cite{CaSc2009} contradicts Proposition 2.1, p. 185, in
\cite{Abate2001}. As a matter of fact, Lemma 3.5 (and thus Theorem 3.6) is not
correct. This is due to the authors misinterpretation of the proof of
Corollary 3.1 in \cite{Abate2001} wrongly stated as Theorem 3.2 in
\cite{CaSc2009}. Indeed, the correct statement  is the following:
{\it Let $G\in\operatorname*{Diff}%
_{1}(\mathbb{C}^{2},0)$ and suppose that $S:=\operatorname*{Fix}(G)$ is a
smooth curve through the origin such that $\operatorname*{ind}_{0}%
(G,S)\notin\mathbb{Q}^{+}$. Then $G$ admits $\nu(f)-1$ parabolic curves.}

More precisely, one can check that this would be the appropriate hypothesis
looking to the proof of Theorem 3.1 in \cite{Abate2001}. Now one can
check that the diffeomorphism in the proof of Lemma 3.5 in \cite{CaSc2009}
does not satisfy the conditions of the above theorem.
}

\end{remark}
}

\noindent\textbf{Acknowledgement}. {\rm The authors would like to thank  Javier Ribon, Marco Abate, Fabio Brochero Mart\'inez and Alcides Lins
Neto for helpful suggestions and comments.}

\vglue.1in

\begin{tabular}
[c]{ll}%
Leonardo Câmara & \quad Bruno Scárdua\\
Departamento de Matemática - CCE & \quad Instituto de Matemática\\
Universidade Federal do Espírito Santo & \quad Universidade Federal do Rio de
Janeiro\\
Av. Fernando Ferrari 514 & \quad Caixa Postal 68530\\
29075-910 - Vitória - ES & \quad21.945-970 Rio de Janeiro-RJ\\
BRAZIL & \quad BRAZIL\\
leonardo.camara@ufes.br & \quad scardua@im.ufrj.br
\end{tabular}

\end{document}